\documentclass{amsart}
\usepackage{amsmath}
\usepackage{amsfonts}
\usepackage{amssymb}
\usepackage{graphicx}
\newtheorem{Theorem}{Theorem}[section]

\newtheorem{Lemma}[Theorem]{Lemma}

\newtheorem{Definition}[Theorem]{Definition}

\newtheorem{example}[Theorem]{Example}
\newcommand{\N}{\mathbb N}

\title[Domain of implicit functions]{On the domain of implicit functions in a projective limit setting without additionnal norm estimates}
\author{Jean-Pierre Magnot}
\address{LAREMA, Universit\'e d’Angers, 2 Bd Lavoisier
	, 49045 Angers cedex 1, France and Lyc\'ee Jeanne d'Arc, 40 avenue de Grande Bretagne, 63000 Clermont-Ferrand, France}
\email{jean-pierr.magnot@ac-clermont.fr}
\begin{document}

\maketitle

\begin{abstract}
We examine how implicit functions on ILB-Fr\'echet spaces can be obtained without metric or norm estimates which  are classically assumed. We obtain implicit functions defined on a domain $D$ which is not necessarily open, but which contains the unit open ball of a Banach space. The corresponding inverse functions theorem is obtained, and we finish with an open question on the adequate (generalized) notion of differentiation, needed for the corresponding version of the Fr\"obenius theorem.
\end{abstract}


\section*{Introduction}
Classical inverse functions theorems, implicit functions theorems and Fr\"obenius theorems on Banach spaces are known to be equivalent. There exists numerous extensions to setting on Fr\'echet or locally convex spaces, and to our knowledge almost all proofs are based on a contraction principle.  In order to obtain in the proofs a mapping which is contracting, one needs to assume conditions which are not automatically fulfilled by a mapping on Fr\'echet spaces, but which are automatically locally fulfilled (on an open set) by a sufficiently regular mapping on Banach spaces.For classical statements, one can see \cite{Dieu,Gl2006,Ham,HN1971,KP2002,Om,Pen}.

We analyze here how a very classical proof of the implicit function theorem can be adapted on a ILB setting, more precisely on a slightly more general framefork, that is when Fr\'echet spaces $E_\infty$ considered are projective limits of a sequence of Banach spaces $(E_i)$ without assumption of density for teh inclusion $E_{i+1}\subset E_i,$ and when the functions $f_\infty$ on Fr\'echet spaces are restrictions of bounded functions $f_i$ on the sequence of Banach space. This is what one may call order 0 maps, by analogy with the order of differential operators. Then we get (Theorem \ref{implicitdeg0}) an implicit function which is defined on a domain $D,$ which is not a priori open in the Fr\'echet topology, but which contains the open ball of a Banach space. This results can be adapted to some functions $f_\infty$ for which there does not exist any extension to a Banach space. These functions have to be controlled by a family of injective maps $\{\Phi_x\}$, which explains the terminology "tame" (Theorem \ref{phi-implicit}). As a special case of applications, we recover the maps $f_\infty$ which extend to maps $E_i \rightarrow E_{i-r}$ (called order $r$ maps)     

We have to remark that the domain $D$ can be very small.
This is the reason why regularity results on implicit functions cannot be stated: differentiability, in a classical sense, requires open domains or at least manifolds. This leads to natural questions for the adequate setting for analysis beyond the Banach setting. Even if not open, following the same motivations as the ones of  Kriegl and Michor in \cite{KM} when they consider smoothness on non open domains, the domain $D$ may inherit some kind of generalized setting for differential calculus, such as diffeologies \cite{Igdiff} which are used in \cite{Ma2016-3}. This question is left open, because out of the scope of this work: the most adapted (generalized) framework for the extension of the regularity (i.e. differentiability) has to be determined.

We then give consequences for an inverse functions theorem, which can be stated with the same restrictions as before on the nature of teh domain $D$, and with an obstruction to follow the classical proof of the Fr\"obenius theorem from \cite{Pen} where differentiation on $D$ is explicitely needed. Finally, in last section, we show how this theorem can describe a Banach Lie gbgroup of a topological group arising in the ILH setting. 
\section{Implicit functions from Banach spaces to projective limits}\label{impl}

Let $(E, ||.||_E) $ anf $(F, ||.||_F)$ be two Banach spaces. The Banach space $E \times F$ is endowed wthh ine norm $||(x,y)||_{E\times F} = \max\{||x||_E,||y||_F\}.$ We note by $D_1$ and $D_2$ the (Fr\'echet) differential with respect to the variables in $E$ and $F$ respectively. Let us first give the statement and a proof of a classical implicit function theorem on Banach spaces, for the sake of extracting key features for generalization. 
For this, let $U$ be an open neighborhood of $O$ in $E,$ let $V$ be an open neighborhood of $0$ in $F,$ and let \begin{eqnarray} \label{fbanach}
f : U \times V & \rightarrow& F
\end{eqnarray}
be a $C^r$-function $(r\geq 1)$ in the Fr\'echet sense, such that
\begin{eqnarray}\label{0banach}
f(0,0)&=&0
\end{eqnarray}
and \begin{eqnarray} \label{D2banach} D_2f(0;0) & = & Id_F. \end{eqnarray}

\begin{Theorem} \label{implicitBanach}
	There exists a constant $c>0$ such that, on the open ball $B(0,c) \subset E,$ there is an unique map $$u : B(0,c) \rightarrow V$$ such that \begin{equation}
	\forall x \in B(0,c), \quad f(x,u(x)) = 0.
	\end{equation}  
\end{Theorem}
Let us remark that regularity if the function $u$ is ignored for it is not necessary for next developments.
We now divide the main arguments of the classical direct proof of this theorem into three lemmas.

\begin{Lemma}
	There exists $c_0>0$ and $K>0$ such that $$||(x,y)||_{E\times F} < c_0 \Rightarrow ||D_1f(x,y)||_{L(E,F)} <K.$$
\end{Lemma}
\begin{proof}
	Since $f$ is $C^1,$ there exists a neighborhood of $(0,0) \in E\times F$ such that $||D_1f(x,y)||_{L(E,F)}$ is bounded.
	\end{proof}
\begin{Lemma} \label{lemma1banach}
	
	There exists $c_1>0$ such that \begin{eqnarray}
	||(x,y)||_{E \times F}<c_1 & \Rightarrow & ||D_2f(x,y) - Id_F||_{L(F)} < \frac{1}{2}.
	\end{eqnarray} 

\end{Lemma}
\begin{proof}
	The map $f$ is of class $C^r$, with $r \geq 1,$ so that, the map $D_2f:(x,y) \in U\times V \mapsto D_2f(x,y)(.) \in L(F)$ is of class $C^{r-1},$ and in particular of class $C^0.$ By the way, $$\exists c_1>0 , \, ||(x,y)||_{E \times F}<c_1  \Rightarrow ||D_2f(x,y) - D_2f(0;0)||_{L(F)} < \frac{1}{2}.$$ 
	\end{proof}
\begin{Lemma} \label{lemma2banach}
	Let $c_1$ be teh constant of Lemma \ref{lemma1banach}. There exists $c_2>0$ such that
	\begin{eqnarray}
	||x||_{E}<c_2 & \Rightarrow & ||f(x,0) ||_F < \frac{c_1}{4}.
	\end{eqnarray} 
\end{Lemma}
\begin{proof}
	The map $f$ is of class $C^r$, with $r \geq 1,$ so that it is in particular of class $C^0.$ By the way, there exists a constant $c_2$ such that $$ ||x-0||_E < c_2 \Rightarrow ||f(x,0) - f(0,0)||_F < \frac{c_1}{4}.$$
\end{proof}
\begin{Lemma} \label{lemma3banach}
	Let $c = \min\{c_0,c_1,c_2,1\}.$ Let $x$ such that $$||x||<c.$$ Then the sequence  $(u_n)_\N \in F^\N,$ defined by induction by \begin{equation}\left\{ \begin{array}{l} u_0 = 0 \\ \forall n \in \N, \, u_{n+1} = u_n -f(x,u_n)\end{array} \right. \end{equation} is well-defined and converges to $u(x) \in V.$
\end{Lemma}

\begin{proof}
	Let us assume $x$ fixed. Let $g(x,y)= y-f(x,y).$ By the way, $u_n = g^n (u_0).$
	Applying Lemma \ref{lemma1banach}, let $(y,y') \in F^2$ such that both $(x,y)$ and $(x,y)$ are in $B(0,c)\subset B(0,c_1).$ $$f(x,y) -f(x,y')= \int_0^1 D_2f(x,ty+(1-t)y').(y'-y) dt$$
	By the way 
	\begin{eqnarray*}
||g(x,y) -g(x,y') ||_F & \leq & \int_0^1 ||D_2g(x,ty+(1-t)y').(y'-y)||_F dt\\
& \leq &  \int_0^1 ||D_2f(x,ty+(1-t)y')-Id_F||_{L(F)}. ||y'-y||_F dt \\
& \leq & \frac{||y'-y||_F}{2} 
	\end{eqnarray*} 
By the way, $g$ is $\frac{1}{2}-$Lipschitz. 
Thus, applying Lemma \ref{lemma2banach},
we obtain by induction:
 $$||u_1 - u_0||_F < \frac{c_1}{4} \Rightarrow ||u_{n+1}-u_n||_F < \frac{c_1}{2^{n+2}}.$$
and  $$\forall n \in \N,|u_n||_F \leq \frac{c_1(2^{n}-1)}{2^{n+2}}<c_1.$$
\vskip 6pt
Hence $(u_n)$ is converging to $u(x),$ which is in $V.$

	\end{proof}
\begin{proof}[Proof of Theorem \ref{implicitBanach}.]
	By Lemma \ref{lemma3banach} the function $x \mapsto u(x)$ exists for $||x||_E < c.$

	\end{proof}
We now adapt these results to the following setting. 
Let $(E_i)_{i\in \N} $ and $(F_i)_{i \in \N}$ be two decreasing sequences of Banach spaces, i.e. $\forall i \in \N,$ we have $E_{i+1}\subset E_i$ and $F_{i+1}\subset F_i,$ with continuous inclusion maps.
 We then consider $U_0$ and $V_0$ two open neighborhoods of $0$ in $E_0$ and $F_0$ respectively, and a function $f_0,$ of class $C^r$ with the same properties  as in equations (\ref{fbanach},\ref{0banach},\ref{D2banach}). Let us now define, for $i \in \N, $ $U_i = U_0 \cap E_i$ and $V_i = V_0 \cap V_i,$ and let us assume that $f_0$ restricts to $C^r-$maps   \begin{eqnarray} \label{fibanach}
f_i : U_i \times V_i & \rightarrow& F_i.
\end{eqnarray}  
Let $E_\infty = \underleftarrow{\lim}\{E_i; \, i \in \N\},$  let $F_\infty = \underleftarrow{\lim}\{F_i; \, i \in \N\}$ and let $f_\infty = \underleftarrow{\lim}\{F_i; \, i \in \N\}.$

\begin{Theorem} \label{implicitdeg0}
	There exists a non-empty domain $D_\infty\subset U_\infty,$ possibily non-open in $U_\infty,$ and a function $u_\infty:D_\infty \rightarrow V_\infty$ such that, 
	$$\forall x \in D_\infty, \, f_\infty(x,u_\infty(x))=0,$$
	and such that $D_\infty$ contains the unit ball of the Banach space $B_{f_\infty}\subset E_\infty$ defined as the domain of the norm $$||x||_{f_\infty} = \sup_{i \in \N } \left\{\frac{||x||_{E_i}}{c_i}\right\} .$$
\end{Theorem}

\begin{proof}
	Let $i \in \N.$ We now consider a maximal domain $D_i\subset U_i$ where there exists an unique function $u_i$ such that $$\forall x \in D_i, \, f_i(x,u_i(x))=0.$$
	This domain is non empty since it contains $0\in E_i$ and, applying Theorem \ref{implicitBanach}, there exists a constant $c_i>0$ such that $$||x||_{E_i}<c_i \Rightarrow x \in D_i.$$
	By the way, any maximal $D_i$ is an open neighborhood of $0.$
	By the way, setting $D_\infty$ as the intersection of such a family $(D_i),$ we get that $D_\infty $ contains $0 \in E_\infty.$ Of course, $D_\infty$ is not a priori open in th e projective limit topology. However, let $$B_{f_\infty} = \left\{x \in E_\infty \, | \, \sup_{i \in \N } \left\{\frac{||x||_{E_i}}{c_i}\right\} < +\infty\right\}.$$
	This space is a Banach space for the norm $$||.||_{f_\infty} = \sup_{i \in \N } \left\{\frac{||.||_{E_i}}{c_i}\right\}.$$
	Since $$||x||_{f_\infty} < 1 \Leftrightarrow \forall i \in \N, \,||x||_{E_i}<c_i, $$
we get that the open ball of radius 1 in $B_{f_\infty}$ is a subset of $D_\infty,$ which ends the proof.
	\end{proof}
Let us now extend it to a class of functions that we call \textbf{tame}. 
For this, we define the sequences $(E_i)_{i\in \N} $ and $(F_i)_{i \in \N}$ as before, as well as $E_\infty$ and $F_\infty.$ We also define a similar sequence $(G_i)_{i \in \N}$ of Banach spaces and $G_\infty$ the projective limit of this family.
\begin{Definition}
	Let $U_0\times V_0$ be an open neighborhood of $0$ in $E_0\times F_0$ and let $U_\infty = U_0 \cap E_\infty$ and $V_\infty = V_0 \cap F_\infty.$ Let us fix $\{\Phi_x\}_{x\in U_\infty}$ be a family of injective maps from $G_\infty to F_\infty.$ A map $$f: U_\infty \times V_\infty \rightarrow G_\infty$$
	is \textbf{$\Phi$-tame} if and only if $$f_\infty = \Phi_x \circ f : U_\infty \times V_\infty \rightarrow V_\infty $$
	extends to $C^r-$maps ($r\geq 1$) 
	$$f_i : U_i \times V_i \rightarrow V_i $$
\end{Definition}

\begin{example}
	If there exists a linear isomorphism $A: E_0 \rightarrow E_1$ which restricts to isomorphisms $E_i \rightarrow E_{i+1},$ setting $\Phi_x = A^r,$ the family of tame maps are exactly the family of functions $f_\infty$ which extend to $C^r-$maps $E_i\rightarrow E_{i-r}.$
\end{example}
\begin{Theorem}\label{phi-implicit}
	Let $f$ be a $\Phi-$tame map, such that, $\forall i \in \N,$ $$D_2f_i(0,0)= Id_{F_i}.$$
	Then there exists a non-empty domain $D_\infty\subset U_\infty,$ possibily non-open in $U_\infty,$ and a function $u_\infty:D_\infty \rightarrow V_\infty$ such that, 
	$$\forall x \in D_\infty, \, f(x,u_\infty(x))=0,$$
Moreover, there exists a sequence of positive real numbers $(c_i)$ such that $D_\infty$ contains the unit ball of the Banach space $B_{f,\Phi}\subset E_\infty$ defined as the domain of the norm $$||x||_{f,\Phi} = \sup_{i \in \N } \left\{\frac{||x||_{E_i}}{c_i}\right\} .$$ 
\end{Theorem}

\begin{proof}
	We apply Theorem \ref{implicitdeg0} to $\Phi \circ f.$ Then, since $\forall x , \Phi_x$ is injective, $\Phi_x \circ f(x,.) = 0 \Leftrightarrow f(x,.)=0.$  
	\end{proof}
\section{Tentatives for inverse functions and Frobenius theorem}
\subsection{``Local'' inverse theorem}
Let $(E_i)_{i\in \N} $ be an decreasing sequence of Banach spaces, i.e. $\forall i \in \N,$ we have $E_{i+1}\subset E_i$ with continuous inclusion maps and let $E_\infty$ be the projective limit of teh family $(E_i)_{i\in \N} .$ Let  $U_0$ be an open neighborhood of $0$ in $E_0 ,$
 and define for $i \in \N\cup \{\infty\}, $ $U_i = U_0 \cap E_i.$ Let $V_0$  be an open neighborhood of $0$ in $E_0 ,$
 and define for $i \in \N\cup \{\infty\}, $ $V_i = V_0 \cap E_i.$ Let $f_\infty : U_\infty \rightarrow V_\infty$ be a $C^r-$map ($r\geq 1$) such that $f(0)=0$, which extends to $C^r-$maps $f_i:U_i \rightarrow V_i$ and such that $Df_i(0) = Id_{E_i}.$
\begin{Theorem}
	There exists a domain $D \subset U_\infty,$ which contains the open unit ball of a Banach space 
	$B_{f_\infty} \subset E_\infty,$ with norm defined by a sequence $(k_i)$ of positive numbers by $$||.||_{f_\infty} = \sup_{i \in \N } \left\{\frac{||.||_{E_i}}{k_i}\right\}$$
	 such that $f_\infty|_D$ is a bijection $D \subset U_\infty \rightarrow f_\infty(D) \subset V_\infty.$
\end{Theorem}

\begin{proof}
	We apply Theorem \ref{implicitdeg0} to $g(x,y) = x - f_\infty(y),$ for $(x,y) \in V_\infty \times U_\infty.$ 
	Indeed, we define a $C^r-$ map 
	$g:V_\infty \times U_\infty \rightarrow E_\infty$ which extends to the maps $$g_i:(x,y) \in V_i \times U_i \mapsto x-f_i(y)\in E_i.$$ We have that $D_2g_i(0;0 ) = Df_i(0) = Id_{E_i},$ so that there exists a domain
	$D_\infty \subset V_\infty,$ and a sequence $(c_i)$ of positive real numbers such that $D_\infty$ contains the unit open ball of the Banach space $B_{g_\infty}\subset E_\infty$ with norm $||.||_{g_\infty} = \sup_{i \in \N } \left\{\frac{||.||_{E_i}}{c_i}\right\}$  and a function $u_\infty : D_\infty \rightarrow U_\infty$ such that $$ \forall x \in D_\infty,\,  x - f_\infty (u_\infty(x))=0.$$
	We set $D = u_\infty(D_\infty).$
	Since each $f_i$ is a $C^0$-map, there exists a sequence $(k_i)$ of positive numbers such that 
	$$||x ||_{E_i} < k_i \Rightarrow ||f_i(x)||_{E_i}<c_i.$$
	By the way, $\forall x \in U_\infty,$ $$\sup_{i \in \N} \left\{ \frac{||x||_{E_i}}{k_i}\right\}<1 \Rightarrow  \sup_{i \in \N} \left\{ \frac{||f_\infty(x)||_{E_i}}{c_i}\right\}<1 \Rightarrow  f_\infty(x) \in D_\infty \Rightarrow x = u_\infty\circ f_\infty(x) \in D.$$
	\end{proof}
\subsection{An obstruction for a  Frobenius theorem}
%We set the following notations, following \cite{Om}:
%Let  $\hbox{\bf E} = (E_i)_{i \in I}$ and  
%$\hbox{\bf F} = (F_i)_{i \in I}$ be two ILH vector spaces, 
%let $O$ be an open subset of $E_0 \times F_0$,  
%$\hbox{\bf O} = (O_i)_{i \in I}$ with $O_i = O \cap ({E_i \times F_i})$.

A setting for an adapted Frobenius theorem would be teh following:
%\begin{Theorem}\label{lFrob}
{\it	
Let 
	$$ f_i : O_i \rightarrow L(E_i,F_i), \quad i \in -\N$$ 
	be a collection of smooth maps satisfying the following condition: 
	$$ i > j \Rightarrow f_j|_{O_i} \hbox{ restricts as a linear map to } f_i$$ and such that, $$\forall (x,y) \in O_i, \forall a,b \in E_i$$
	$$(D_1f_i(x,y)(a)(b) + (D_2f_i(x,y))(f_i(x,y)(a))(b) =$$
$$(D_1f_i(x,y)(b)(a) + (D_2f_i(x,y))(f_i(x,y)(b))(a) $$
	(this condition is the analogous of the Frobenius condition in a Banach setting, that we call the ILH Frobenius condition).
	
	Then,
		$\forall (x_0, y_0) \in O_{\infty}$, there exists Fr\"olicher space $ D $ that contains $(x_0, y_0)$ and a smooth map
	$J : D \rightarrow \hbox{\bf F}$
	such that (conditions linked to differentiability of J) 
%	$ \forall (x,y) \in D$ and for each smooth path $c$  on $O_\infty$ with range on $D,$ with $c(0) = (x,y),$ we have $$ D_1(J \circ c) = f_i(x, J(x,y)). c'(0). $$
%	and, if $D_{x_0}=\{(x_0,y) \in D \},$ 
%	$$J_i(x_0,.) = Id_{D_{x_0}}.
%	$$ 
%\end{Theorem}
}

\vskip 12pt
Let us now try to adapt the classical proof in e.g. \cite{Pen}, with the help of theorem \ref{implicitdeg0}.
 We can assume with no restriction that $f(0;0)=0.$
We consider  
$$G_i = C^1_b([0,1],F_i) = \{ \gamma \in C^1([0,1],F_i) | \gamma(0)=0 \}$$ and 
$$ H_i = C^0([0,1],F_i),$$
endowed with their usual topologies. Obviously, if $ i < j$, 
the injections $ G_j \subset G_i$ and $ H_j \subset H_i$ are continuous.

Let us consider $B_0$ an open ball of $E_0$ centered in 
$x_0$,   $B'_0$ an open ball of $F_0$ centered in $y_0$, 
$B''_0$ an open ball of $G_0$ centered in $0$. We set 
$ B_i = B_0 \cap E_i$,  $ B_i' = B_0' \cap F_i$ and 
$ B_i'' = B_0'' \cap G_i$. Then, we define, for $ i \in \N \cup \{\infty\}$,  
$$ g_i : B_i \times B'_i \times B''_i \rightarrow H_i $$
$$ g(x,y,\gamma)(t) = \mathfrak{d \gamma }{ dt}(t) - f_i(t(x-x_0) + x_0, y + \gamma(t)).(x - x_0).$$
We then apply Theorem \ref{implicitdeg0} to $\int_0^{(.)} \circ g_\infty:  B_\infty \times B'_\infty \times B''_\infty \subset (E_\infty \times E_\infty) \times G_\infty \rightarrow G_\infty $. There exists 
a domain $D_\infty$ 
%$U_i,V_i$ open neighborhoods of $x_0$ and $y_0$ in $E_i$ 
%and $F_i$ 
such that we can define the function $\alpha_\infty$ as the unique function such that

$$ \left\{ \begin{array}{l}
\alpha_\infty(x_0) = 0 \\
g_i(x,y,\alpha_\infty(x,y)) = 0 , \quad \forall (x,y) \in D_\infty \\
\end{array}
\right. .$$
%Since, if $i<j$, $g_j|_{B_i \times B_i' \times B_i''} = g_i$, 
%we have that $\alpha_i$ is equal to the restriction of $\alpha_j$. 
%Let us define, by induction, $D_i$ as the maximal domain for inclusion in $D_j \cap E_i$ where the map $\alpha_i$ can be defined.  

Since we set $J(x,y) = y + \alpha_\infty(x,y)(1)$
%, this shows that, 
%if $i<j$, $J_j|_{U_i \times V_i} = J_i$. Hence, setting 
%$D = \bigcap D_i,$ 
%the family $(J_i)_{i \in -\N}$ defines, by projective limit, 
%a unique smooth map $J_\infty : D \rightarrow F_\infty$. 
%Thus, $J = (J_i)_{i \in I}$ is a smooth map. 
Uniqueness follows from Theorem \ref{implicitdeg0}.  

\vskip 12pt
\noindent\textbf{Open problem:} We are now facing a theorical impossibility. Classical theory of differentiation is valid for functions on open domains. We need here to consider $D_1J,$ which is here defined on $D_\infty$ which is not a priori open. There exists numerous extensions of the classical theory of differentiation, one of them is used in \cite{Ma2016-3} based on \cite{Igdiff}. Which one is better for this setting? 

%\vskip 12pt
%\section{Case studies}
%Let us now fix $$E_i = H^i([0;1],\R),$$ equipped with the (classical) Sobolev norms $$||w||_{E_i} = \left(\int_0^1 w.\left(\left(1 + \left(- \frac{d}{dx}^{2}\right)^i \right)w \right)\right)^{1/2}.$$

%Let us now analyze the Banach space $B_{f_\infty}$ described in Theorem \ref{implicitdeg0}, for an arbitrary function $f_\infty$ which fulfils the assumptions of this result.
%\begin{Proposition}
%	For any sequence of functions $f_i$ defined as in section \ref{impl}, the Banach space $B_{f_\infty}$ is infinite dimensional.
%\end{Proposition}
%\begin{proof}
%	Let us recall that, by construction, the family $c_i$ is such that $c_i\leq 1$ by the consequences of Lemma \ref{lemma3banach}, and that $c_0 = 1.$ 
	
	%Let $w \in \R[x],$ of order $o.$ Then $$\forall i > \frac{o}{2},\left(\left(1 + \left(- %\frac{d}{dx}^{2}\right)^i \right)w = w,$$ and hence $||w||_{E_i} = ||w||_{E_0}$ which shows that 
	%$$\sup_{i \in \N } \left\{\frac{||x||_{E_i}}{c_i}\right\} \leq \sup_{i \in \N } \left\{\frac{||x||_{E_i}}{c_i}\right\}
%\end{proof}
%Then we get 
%\begin{Theorem}
%	For any sequence of functions $f_i$ defined as in section \ref{impl},
%	$B_{f_\infty}$ is dense in $C^\infty([0;1], \R).$
%\end{Theorem}
	
	\section{An application of the implicit functions theorem on $\mathcal{L}_\infty$}
	We consider here a sequence of Banach spaces $(E_i)$ as before, and we asssume also that $\forall i, \, E_{i+1}$ is dense in $E_i.$
	Following \cite{DGV,Om}, we consider the set  of linear maps $E_\infty \rightarrow E_\infty$ which extend to bounded linear maps $E_i \rightarrow E_i.$ Let us note it as $\mathcal{L}_\infty$, and $G\mathcal{L}_\infty = \bigcap_{i \in \N}GL(E_i)$ is a group known as a topological group \cite{DGV}, and \cite{Om} quotes ``natural differentiation rules'' that are identified in \cite{Ma2013,MaW2017} as generating a smooth Lie group for generalized differentiation on Fr\"olicher or diffeological  spaces. Let $i \in \N,$ we define $$ \mathcal{L}_i = \left\{ a \in L(E_i)\,\Big{|}\, a|_{E_\infty}\in \mathcal{L}_\infty\right\}.$$
	We equip these spaces with the norms 
	$$||a||_i = \max\left\{ ||a||_{L(E_{i-r})}\, \Big{|}\, 0 \leq r \leq i \right\}.$$
	We apply Theorem \ref{implicitdeg0} to the map 
	$$f_{\infty}:(a,b) \in \mathcal{L}_\infty^2 \mapsto (Id+a)(Id+b) - Id,$$
	for the sequence of Banach spaces $(\mathcal{L}_i)$ with projective limit $\mathcal{L}_\infty.$
	We already know that the maximal domain $D_\infty$ of the implicit function obtained will be 
	$$D_\infty \supset \left\{ a \in \mathcal{L}_\infty \, \Big{|} \,  Id + a \in G\mathcal{L}_\infty\right\}$$
	and the implicit function will be $$u_\infty: a \in D_\infty \mapsto (Id + a)^{-1} - Id$$ where $(Id + a)^{-1}$ is the left inverse of $Id+a.$
	But the main question about $G\mathcal{L}_\infty$ is the most adequate structure for it: it behaves like a Lie group \cite{Om,Ma2013}, but does not carry a priori charts which allows us only to consider it as a topological group \cite{DGV}. Applying Theorem \ref{implicitdeg0}, there exists a Banach subspace $B$ of $\mathcal{L}_\infty$ defined by the norm
	$$||a|| = \sup_{i \in \N } \left\{\frac{||a||_{_i}}{c_i}\right\} .$$
	But we easily show that each $L_i$ is a Banach algebra, so that, $c_i = 1$ since its group of the units contains te open ball of radius $1$ centrered at $Id.$ By the way, 
	$$ ||a|| =  \sup_{i \in \N } \left\{{||a||_{L(E_i)}}\right\} ,$$
	and $B$ is a Banach algebra, with group of the units $GL(B) \subset G\mathcal{L}_\infty$ which is a Banach Lie group.
	 
	We finish with the special case when $(E_i)$ is a ILH sequence (i.e. a sequence of Hilbert spaces with bounded and dense inclusion, see \cite{Om}) and when there exists a self-adjoint, positive (unbounded) operator $Q$ such that $$(Q^i a,b)_{E_0} = (a,b)_{E_i}.$$ In this case, there exists $(e_k)_{k \in \N}$ an orthonormal base in $E_0$ of eigenvectors of $Q$ in  $E_\infty$ with is also orthogonal in $E_i.$ In this case, the orthogonal projections $$a \mapsto (e_k,a)_{E_0}e_k$$ restrict to operators in $B,$ which shows that $B$ is an infinite dimensional Banach algebra.   
	
	\vskip 12pt
	
	\noindent
	\textbf{Open question:} There is a natural right action of $GL(B)$ on $G\mathcal{L}_\infty$ by composition. What is the structure of $G\mathcal{L}_\infty / GL(B)$?

\end{document}